\documentclass{pja00b}

\theoremstyle{plain}
\newtheorem{theorem}{Theorem}[section]
\newtheorem{lemma}[theorem]{Lemma}
\newtheorem{proposition}[theorem]{Proposition}

\theoremstyle{definition}
\newtheorem{definition}[theorem]{Definition}
\newtheorem{remark}[theorem]{Remark}

\runninghead{Integrable Teichm\"uller space}
\title{The $p$-integrable Teichm\"uller space for $p \geqslant 1$}
\Author{1}{Huaying}{Wei} \Author{2}{Katsuhiko}{Matsuzaki}

\affiliation{1}{Department of Mathematics and Statistics, Jiangsu Normal University, Xuzhou 221116, PR China}
\affiliation{2}{Department of Mathematics, School of Education, Waseda University, Shinjuku, Tokyo 169-8050, Japan}

\KeyWords{{universal Teichm\"uller space}{Weil--Petersson Teichm\"uller space}{Bers embedding}{bi-Lipschitz quasiconformal extension}{analytic Besov space}}

\Subject [2020]{30C62, 30H25, 32G15} 

\begin{document}
\maketitle
\begin{abstract}
We verify that the $p$-integrable Teichm\"uller space $T_p$ admits the canonical complex Banach manifold structure for
any $p \geq 1$. Moreover, we characterize a quasisymmetric homeo\-morphism corresponding to
an element of $T_p$ in terms of the $p$-Besov space for any $p>1$.
\end{abstract}

\section{Introduction}

The {\it universal Teichm\"uller space} $T$ is a certain quotient space of
all normalized quasiconformal 
homeomorphisms of the unit disk $\mathbf D$
(or the upper half-plane $\mathbf U$). This is the total space of all other 
Teichm\"uller spaces defined by means of quasiconformal mappings.
See \cite{Ah, Le, Na}. 

A quasiconformal map $f$ of $\mathbf D$
into the complex plane $\mathbf C$ is
determined uniquely by its {\it complex dilatation} 
$\mu_f(z)=f_{\bar z}/f_z$
up to the post composition of a conformal map.
Then, the space $M(\mathbf D)$ of all {\it Belt\-rami coefficients} $\mu$, which are measurable functions
with $\Vert \mu \Vert_\infty<1$, is identified with the space of normalized 
quasiconformal maps on $\mathbf D$.
Let $f^{\mu}:\mathbf D \to \mathbf D$ denote
the quasiconformal homeomorphism 
with complex dilatation $\mu \in M(\mathbf D)$ normalized by fixing boundary points $1$, $-1$, and $-i$.

We say that $\mu$ and $\nu$ in $M(\mathbf D)$ are {\it Teichm\"uller equivalent} if
the extensions of $f^{\mu}$ and $f^{\nu}$ to the unit circle $\mathbf S$, which are {\it quasisymmetric homeomorphisms}, are the same.
Then, $T$ is the quotient space by this equivalence relation;
the quotient map $\pi:M(\mathbf D) \to T$ $(\pi(\mu)=[\mu])$ is called the
{\it Teichm\"uller projection}. The topology on $T$ is induced by the Teichm\"uller distance (see \cite[Sect. III.2]{Le}); equivalently,
$[\mu]$ converges to $[\nu]$ in $T$ if 
$$
\inf\,\{\Vert \mu \ast \nu^{-1}\Vert_\infty \mid \mu \in [\mu],\ \nu \in [\nu]\} \to 0,
$$
where $ \mu \ast \nu^{-1}$ stands for the complex dilatation of $f^\mu \circ (f^{\nu})^{-1}$.
We call this the {\it Teichm\"uller topology}.

Integrable Teichm\"uller spaces are subspaces of $T$
given by integrable Beltrami coefficients $\mu$ with respect to the hyperbolic metric on $\mathbf D$.

\begin{definition}
For $\mu \in M(\mathbf D)$ and $p>0$, let
$$
\Vert \mu \Vert_p=\left(\int_{\mathbf D} |\mu(z)|^p(1-|z|^2)^{-2}dxdy\right)^{\frac{1}{p}}.
$$  
The space of all Beltrami coefficients $\mu \in M(\mathbf D)$ with $\Vert \mu \Vert_p<\infty$
is denoted by $M_p(\mathbf D)$.
The {\it $p$-integrable Teichm\"uller space} $T_p$ is defined by $\pi(M_p(\mathbf D))$.
For $p \geq 1$, we regard $M_p(\mathbf D)$ 
as the open subset 
of the Banach space of norm $\Vert \mu \Vert_\infty+\Vert \mu \Vert_p$.
We equip $T_p$ with the Teichm\"uller topology induced by this norm. 
\end{definition}

For $p=2$, $T_2$ is also known as the Weil--Petersson Teichm\"uller space. This was
first introduced by Cui \cite{Cu} and developed by Takhtajan and Teo \cite{TT}.
It is equipped with a complex Hilbert mani\-fold structure and
with an invariant K\"ahler metric called the Weil--Petersson metric.
Shen \cite{Sh} 
characterized the quasisymmetric extension $h$ of a quasiconformal 
self-homeomorphism $f$ of $\mathbf D$ with $\mu_f \in M_2(\mathbf D)$
by the condition that $h$ is 
absolutely continuous and $\log h'$ belongs to the Sobolev space $H^{\frac{1}{2}}(\mathbf S)$.

For $p > 2$, $T_p$ was first considered by Guo \cite{Gu}. The complex Banach
manifold structure for $T_p$ was provided by Yanagishita \cite{Ya} including the case  
with the Fuchsian group action.
This was also done by Tang and Shen \cite{TS} 
where they also 
generalized the characterization of the quasisymmetric extension $h$ to $\mathbf S$
of a quasiconformal self-homeomorphism 
$f$ with $\mu_f \in M_p(\mathbf D)$. 
The $p$-Weil--Petersson metric was 
introduced and investigated in
\cite{Mat} and \cite{Ya2}.

In contrast, the case of $0 <p \leq 1$ was studied by Alberge and Brakalova \cite{AB} and
they in particular proved that the quasisymmetric extension $h$ is continuously differentiable in this case.

In this paper, we consider the case of $1 \leq p<2$ and prove that
$T_p$ is also endowed with the complex Banach manifold structure 
(Theorem \ref{Pcomplex}).
The results formulated on $\mathbf D$ and $\mathbf S$ can be suitably translated into
the upper half-plane $\mathbf U$ and the real line $\mathbf R$.
We also prove that 
a quasisymmetric homeomorphism $h:\mathbf R \to \mathbf R$ has 
a quasiconformal extension $f$ to $\mathbf U$ with $\mu_f$
%a quasiconformal self-homeomorphism of $\mathbf U$
%has its complex dilatation 
in $M_p(\mathbf U)$ for $p>1$ if and only if 
%its quasisymmetric extension $h$ to $\mathbf R$ 
$h$ is locally absolutely continuous
and $\log h'$ is in the $p$-Besov space $B_p(\mathbf R)$
(Theorem \ref{Besov}). The claim for $M_p(\mathbf D)$ and $B_p(\mathbf S)$ is also verified.
These results improve those in the recent papers \cite{LiS,LS}.

\section{The Bers embedding}

For $\mu \in M(\mathbf D)$, let $f_{\mu}$ be 
the quasiconformal self-homeo\-morphism of the extended complex plane $\widehat{\mathbf C}$ such that
its complex dilatation is $\mu$ on $\mathbf D$ and $0$ on the exterior disk 
$\mathbf D^*=\{z \in \mathbf C \mid |z|>1\} \cup \{\infty\}$.
We take the Schwarzian derivative $S_{f_\mu|_{\mathbf D^*}}$ of the conformal homeo\-morphism $f_\mu$ on $\mathbf D^*$.
This belongs to the complex Banach space $\mathcal A_\infty(\mathbf D^*)$ of 
all holomorphic maps
$\varphi$ on $\mathbf D^*$ with 
$
\Vert \varphi \Vert_{\mathcal A_\infty}=\sup_{z \in \mathbf D^*}\, (|z|^2-1)^2|\varphi(z)|<\infty.
$

A map $\Phi:M(\mathbf D) \to \mathcal A_\infty(\mathbf D^*)$ 
is defined by
$\mu \mapsto S_{f_\mu|_{\mathbf D*}}$, which we call the Bers Schwarzian derivative map.
Then, $\pi(\mu_1)=\pi(\mu_2)$ if and only if $\Phi(\mu_1)=\Phi(\mu_2)$ for $\mu_1,\mu_2 \in M(\mathbf D)$. 
Hence, there is
a well-defined injection $\beta:T \to \mathcal A_\infty(\mathbf D^*)$ such that $\beta \circ \pi=\Phi$,
which is called the {\it Bers embedding}.

The following property of $\Phi$ is  
well known.
See \cite[Th.V.5.3]{Le}, \cite[Sect.3.4, 3.5]{Na}.

\begin{proposition}\label{submersion}
$\Phi:M(\mathbf D) \to \mathcal A_\infty(\mathbf D^*)$ 
is a holomorphic split submersion onto
the image.
\end{proposition}

This proposition in particular implies (a) $\Phi$ is continuous; (b) 
$\Phi$ has a local continuous right inverse at every point in the image.
By the formula
\begin{equation}\label{chain}
\mu \ast \nu^{-1}(f^\nu(z))=\frac{\mu(z)-\nu(z)}{1-\overline{\nu(z)}\mu(z)}\cdot\frac{\partial f^\nu(z)}{\overline{\partial f^\nu(z)}},
\end{equation}
we see that the quotient topology on $T$ induced by $\pi$ coincides with the Teichm\"uller topology.
Then, (a) is equivalent to saying that $\beta$ is continuous, and (b) implies that
$\beta^{-1}$ is continuous.
Hence, we have:

\begin{proposition}\label{Tcomplex}
The Bers embedding $\beta:T \to \mathcal A_\infty(\mathbf D^*)$ is a homeomorphism onto 
the open subset 
$\Phi(M(\mathbf D))$. Thus, $T$ is endowed with
the complex structure modeled on $\mathcal A_\infty(\mathbf D^*)$.
The Teichm\"uller projection $\pi:M(\mathbf D) \to T$ is holomorphic with
a local holomorphic right inverse at every point of $T$.
\end{proposition}

We restrict $\Phi$ to $M_p(\mathbf D)$.
Let
\begin{align*}
{\mathcal A}_p(\mathbf D^*)&
=\{\varphi \in \mathcal A_\infty(\mathbf D^*) \mid \Vert \varphi \Vert_{\mathcal A_p}<\infty\};\\
\Vert \varphi \Vert_{\mathcal A_p}&=\left(\int_{\mathbf D^*} |\varphi(z)|^p(|z|^2-1)^{2p-2}dxdy\right)^{\frac{1}{p}}.
\end{align*}
For $p \geq 1$, ${\mathcal A}_p(\mathbf D^*)$ is a complex Banach space with this norm.
There is a constant $c_p>0$ depending only on $p$ such that
$\Vert \varphi \Vert_{\mathcal A_\infty} \leq c_p \Vert \varphi \Vert_{\mathcal A_p}$.
We will see in Lemma \ref{TTholo} below that $\Phi(M_p(\mathbf D)) \subset {\mathcal A}_p(\mathbf D^*)$.

We also consider the Bers embedding $\beta$ on the $p$-integrable Teichm\"uller space $T_p=\pi(M_p(\mathbf D))$.
Then, $\beta:T_p \to  {\mathcal A}_p(\mathbf D^*)$ is an injection onto $\Phi(M_p(\mathbf D))$.
The topology on $T_p$ is the Teichm\"uller topology induced from $M_p(\mathbf D)$ in which $[\mu]$ converges to $[\nu]$ in $T_p$ if
$$
\inf\,\{\Vert \mu \ast \nu^{-1}\Vert_p+\Vert \mu \ast \nu^{-1}\Vert_\infty \mid \mu \in [\mu],\ \nu \in [\nu]\} \to 0.
$$
This turns out to be the same as the topology defined by replacing $\Vert \mu \ast \nu^{-1}\Vert_p+\Vert \mu \ast \nu^{-1}\Vert_\infty$ with
$\Vert \mu \ast \nu^{-1}\Vert_p$.

\section{Bi-Lipschitz quasiconformal maps}

We show that there is some $\nu \in M_p(\mathbf D)$ in every Teich\-m\"uller class $[\mu] \in T_p$
such that $f^\nu$ is a bi-Lipschitz 
self-diffeomorphism of $\mathbf D$ in the hyperbolic metric.
We adapt several claims in Takhtajan and Teo \cite{TT} for $p=2$ to the general case $p \geq 1$.

\begin{lemma}\label{group}
Let $\mu, \nu \in M_p(\mathbf D)$ for $p \geq 1$ and assume that
$f^\nu$ is a bi-Lipschitz self-homeomorphism 
of $\mathbf D$ in the hyperbolic metric.
Then, $\mu \ast \nu^{-1}$ belongs to $M_p(\mathbf D)$. 
In particular, the complex dilatation $\nu^{-1}$ of $(f^{\nu})^{-1}$ is in $M_p(\mathbf D)$. 
Moreover,
$\mu$ converges to $\nu$ in $M_p(\mathbf D)$ if and only if
$\mu \ast \nu^{-1}$ converges to $0$.
\end{lemma}

\begin{proof} By formula \eqref{chain} and change of variables, 
\begin{align*}
&\Vert \mu \ast \nu^{-1} \Vert_p^p=\int_{\mathbf D}
\left|\frac{\mu-\nu}{1-\bar \nu \mu}\right|^p \circ (f^{\nu})^{-1}(z)\frac{dxdy}{(1-|z|^2)^2}\\
&\leq C
\int_{\mathbf D}|\mu(\zeta)-\nu(\zeta)|^p\frac{d\xi d\eta}{(1-|\zeta|^2)^2}=C\Vert \mu-\nu \Vert^p_p,
\end{align*}
and $\Vert \mu \ast \nu^{-1} \Vert_p^p \geq C^{-1}\Vert \mu-\nu \Vert^p_p$
for a constant $C \geq 1$ depending only on $\Vert \nu \Vert_\infty$ and the bi-Lipschitz constant of $f^\nu$.
\end{proof}

By this lemma,
any $\nu \in M_p(\mathbf D)$ such that $f^\nu$ is bi-Lipschitz defines a map $r_\nu:M_p(\mathbf D) \to M_p(\mathbf D)$ by
$r_\nu(\mu)=\mu \ast \nu^{-1}$ for $\mu \in M_p(\mathbf D)$. A similar computation shows that
there is a constant $C >0$ with the same dependence as above such that
$$
\Vert r_\nu(\mu_1)-r_\nu(\mu_2) \Vert_p \leq C \Vert \mu_1 -\mu_2 \Vert_p
$$
for any $\mu_1, \mu_2 \in M_p(\mathbf D)$. The same inequality is true for $\Vert \cdot \Vert_\infty$.
Hence, $r_\nu$ is continuous. By considering $(r_\nu)^{-1}=r_{\nu^{-1}}$,
we see that $r_\nu$ is a homeomorphic (in fact biholomorphic) automorphism of $M_p(\mathbf D)$.

\begin{lemma}\label{TTholo}
For $p \geq 1$, 
there exists a constant $C>0$ such that
$\Vert \Phi(\mu) \Vert_{\mathcal A_p} \leq C \Vert \mu \Vert_p$
for every $\mu \in M_p(\mathbf D)$.
Moreover, 
%if $f^\nu$ is bi-Lip\-schitz for $\nu \in M_p(\mathbf D)$, then 
$\Phi:M_p(\mathbf D) \to {\mathcal A}_p(\mathbf D^*)$ is holomorphic.
\end{lemma}

\begin{proof}
We can verify the first claim by \cite[Chap.1, Lem.2.9]{TT}.
Although this reference is stated for $p=2$, the argument works
for any $p \geq 1$. 
More explicitly, the first part of the proof of Theorem 2.3 in
\cite{TT} can be modified to $p \geq 1$ if we apply the H\"older inequality to the formula in Page 25.
Then, Remark 2.4 is also true for $p \geq 1$. 
This gives an estimate of the derivative of $\Phi$, and thus the first claim of our lemma follows.
This in particular implies that $\Phi$ is locally bounded.

To show that $\Phi$ is holomorphic, it suffices to see that 
$\Phi$ is G\^ateaux holomorphic.
%(see \cite[p.28]{Bour}).
As in \cite[Lem.V.5.1]{Le}, we can verify this by a usual argument (see Remark \ref{usual} below), for which
we may rely on Proposition \ref{submersion}. 
%Corollary 2.6 is true under the assumption that $f^\nu$ is bi-Lip\-schitz,
%and so is Remark 2.8. 
%Our second claim as well as Lemma 2.9 for $p \geq 1$
%is proved as the application of these generalized
%Remarks 2.4 and 2.8 with Lemma \ref{group}.
\end{proof}

\begin{remark}
In other papers \cite{Cu, Gu, Tang}, the claim that $\Phi:M_p(\mathbf D) \to {\mathcal A}_p(\mathbf D^*)$
is continuous (and hence holomorphic) was proved for $p \geq 2$. This is due to
the integral representation of the Schwarzian derivative
formulated by Astala and Zinsmeister \cite[Form.(4)]{AZ}:  
$$
(|\zeta|^{2}-1)^2|\Phi(\mu)(\zeta)|^2 \leq C\int_{\mathbf{D}}\frac{|\mu(z)|^2} 
{|z-\zeta|^4}dxdy \quad(\zeta \in \mathbf{D}^*)
$$
for $\mu \in M_2(\mathbf D)$,
where $C>0$ is a constant depending only on 
$\Vert\mu\Vert_{\infty}$. We can modify this formula
for the estimate of the difference $\Phi(\mu_1)-\Phi(\mu_2)$ (see \cite[Form.(5.4)]{SW}). 
Moreover, the H\"older inequality yields
the corresponding claim for $p \geq 2$.
\end{remark}

\begin{lemma}\label{biLip}
Each Teichm\"uller class $[\mu] \in T_p$ $(p \geq 1)$ contains $\nu \in M_p(\mathbf D)$
such that $f^\nu$ is
a bi-Lip\-schitz diffeomorphism in the hyperbolic metric with 
the bi-Lipschitz constant depending only on $\Vert \mu \Vert_\infty$.
\end{lemma}

\begin{proof}
If $\Vert \mu \Vert_\infty<1/3$, then the Ahlfors--Weill section 
$\nu=\sigma(\varphi) \in M(\mathbf D)$ defined by
\begin{equation}\label{AW}
\sigma(\varphi)(z^*)=-\frac{1}{2}(zz^*)^2(1-|z|^2)^2\varphi(z)\ \ (z^*=1/\bar z)
\end{equation}
for $\varphi=\Phi(\mu)$ satisfying
$\Vert \varphi \Vert_{\mathcal A_\infty}<2$ is Teichm\"uller equivalent to $\mu$. 
See \cite[Th.II.5.1]{Le}. By Lemma \ref{TTholo}, $\varphi \in {\mathcal A}_p(\mathbf D^*)$.
Then, we see that $\nu$ belongs to $M_p(\mathbf D)$ by formula \eqref{AW}.
Moreover, as in \cite[Chap.1, Lem.2.5]{TT}, $f^\nu$ is 
a bi-Lipschitz self-diffeomorphism of $\mathbf D$ if $\Vert \mu \Vert_\infty<\delta$ for some $\delta \leq 1/3$.

For an arbitrary $\mu \in M_p(\mathbf D)$, 
let $\mu_k=k\mu/n \in M_p(\mathbf D)$ $(k=1,\ldots,n)$, where
$n \in \mathbf N$ is chosen so that $\Vert \mu_{k+1} \ast \mu_k^{-1} \Vert_\infty <\delta$. Suppose we obtain $\nu_k \in M_p(\mathbf D)$ with
$[\nu_k]=[\mu_k]$ and $f^{\nu_k}$ is bi-Lipschitz. Lemma \ref{group} implies $\mu_{k+1} \ast \nu_k^{-1} \in M_p(\mathbf D)$. Hence, we have 
$$
\sigma(\Phi(\mu_{k+1} \ast \nu_k^{-1}))=\sigma(\Phi(\mu_{k+1} \ast \mu_k^{-1})) \in M_p(\mathbf D)
$$ 
as above, which gives a bi-Lipschitz self-diffeo\-mor\-phism of $\mathbf D$. Its composition with $f^{\nu_k}$ is also
bi-Lipschitz, and let $\nu_{k+1}$ be the complex dilatation of this composition. We have $\nu_{k+1} \in M_p(\mathbf D)$ by Lemma \ref{group}, and 
$[\nu_{k+1}]=[\mu_{k+1}]$. By induction, $\nu=\nu_n$ satisfies $[\nu]=[\mu]$ and $f^{\nu}$ is bi-Lipschitz.
\end{proof}

\begin{remark}
For $p \geq 2$, to obtain the bi-Lipschitz diffeomorphism as in the above lemma,
the barycentric extension due to Douady and Earle (see \cite[Th.2]{DE}) was used in 
\cite{Cu, TS, Ya}, and others.
\end{remark}

\section{The complex structure of integrable Teichm\"uller spaces}

We endow $T_p$ for $p \geq 1$ with the complex Banach manifold structure.
The corresponding result to Proposition \ref{Tcomplex} will be proved.
For $p \geq 2$, this was proved in \cite[Th.2.1]{TS} as a consequence of the fact that
$\Phi:M_p(\mathbf D) \to \mathcal A_p(\mathbf D^*)$ is continuous 
by Tang \cite[Th.3.1]{Tang}.
This also implies that the Teichm\"uller topology and
the quotient topology induced by $\pi$ on $T_p$ are the same for $p \geq 2$. 
%However, our result does not assert the continuity of $\Phi$ or $\pi$.
We can extend this to $p \geq 1$ as the following theorem shows.

\begin{theorem}\label{Pcomplex} 
For $p \geq 1$,
the Bers embedding $\beta:T_p \to \mathcal A_p(\mathbf D^*)$ is a homeomorphism onto 
an open subset $\Phi(M_p(\mathbf D))$. Thus, $T_p$ is endowed with
the complex structure modeled on $\mathcal A_p(\mathbf D^*)$.
The Teich\-m\"uller projection $\pi:M_p(\mathbf D) \to T_p$ is holomorphic and has
a local holomorphic right inverse at every point of $T_p$.
\end{theorem}

\begin{proof}
The continuity of $\beta$ follows from Lemmas \ref{group}, \ref{TTholo} and \ref{biLip}.
The fact that 
$\Phi(M_p(\mathbf D))$ is open in $\mathcal A_p(\mathbf D^*)$
is included in the claim that there is a local continuous right inverse $s_\psi$ of $\Phi$ at every point
$\psi \in \Phi(M_p(\mathbf D))$, which is proved 
in Lemma \ref{localsection} below.
Then by Lemma \ref{holomorphy}, $s_\psi$ is in fact holomorphic. 
By Lemma \ref{group}, $\pi:M_p(\mathbf D) \to T_p$ is continuous at $\nu=s_\psi(\psi)$ because 
we choose $f^\nu$ to be bi-Lipschitz.
Since $\beta^{-1}(\psi)=\pi \circ s_\psi(\psi)$,
we see that $\beta^{-1}$ is continuous at $\psi$. The homeomorphism $\beta$ transfers
the claim for $\Phi$ to 
that for $\pi$. 
\end{proof}

\begin{remark}
In the argument of \cite{TT} 
showing the above theorem for $T_2$,
the continuity of $\beta$ is similarly obtained as above,
%a consequence of Lemma \ref{TTholo}, 
but the continuity of $\beta^{-1}$ requires $p=2$ in the second part of the proof of \cite[Chap.1, Th.2.3]{TT}.
\end{remark}

\begin{lemma}\label{localsection}
For every $\psi \in \Phi(M_p(\mathbf D))$ and for every $\nu \in M_p(\mathbf D)$
with $\Phi(\nu)=\psi$ such that $f^\nu$ is bi-Lipschitz,
there is a continuous map $s_\psi:V_\psi \to M_p(\mathbf D)$
on some neighborhood $V_\psi \! \subset \!\mathcal A_p(\mathbf D^*)$
of $\psi$ such that $\Phi \circ s_\psi$ is the identity on $V_\psi$ and $s_\psi(\psi)=\nu$.
\end{lemma}

\begin{proof}
By Lemma \ref{biLip}, for any $\psi \in \Phi(M_p(\mathbf D))$,
we can take $\nu \in B_p(\mathbf D)$ such that $\Phi(\nu)=\psi$ and
$f_{\nu}|_{\mathbf D}$ is a bi-Lipschitz diffeo\-morphism. 
The quasiconformal reflection
$j:f_{\nu}(\mathbf D) \to f_{\nu}(\mathbf D^*)$ over 
the quasicircle $f_{\nu}(\mathbf S)$ is defined by
$j(\zeta)=f_{\nu}(f_{\nu}^{-1}(\zeta)^*)$ for $\zeta \in f_{\nu}(\mathbf D)$. 
Then, there is a constant $c_1>0$ such that
$$
|j_{\bar z}(\zeta)| \rho_{\Omega^*}(j(\zeta)) \leq c_1 \rho_{\Omega}(\zeta), 
$$
where $\rho_\Omega$ and $\rho_{\Omega^*}$ denote the hyperbolic densities of $\Omega=f_{\nu}(\mathbf D)$ and
$\Omega^*=f_{\nu}(\mathbf D^*)$. Moreover,
by \cite[Lem.3]{EN},
there is $c_2>0$ such that
$$
|\zeta-j(\zeta)|^2 \rho_\Omega(\zeta) \rho_{\Omega^*}(j(\zeta)) \leq c_2.
$$
The constants $c_1$ and $c_2$ depend only on $\Vert \nu \Vert_{\infty}$. Hence, for $c=c_1c_2$, we have
\begin{equation}\label{reflection}
|\zeta-j(\zeta)|^2 |j_{\bar z}(\zeta)| \leq c \rho_{\Omega^*}^{-2}(j(\zeta)).
\end{equation}

As in \cite[Sect.6]{EN} and \cite[Sect.II.4.2]{Le}, 
there is a constant $\varepsilon \in (0,1/c)$ depending only on $\Vert \nu \Vert_{\infty}$
such that if $\varphi \in \mathcal A_p(\mathbf D^*)$ satisfies $\Vert \varphi \Vert_{\infty} <\varepsilon$, then
there is a quasiconformal self-homeomorphism $g$ of $\widehat{\mathbf C}$ conformal on 
$\Omega^*$ such that
$S_{g \circ f_{\nu}|_{\mathbf D^*}}=\psi+\varphi$.
In this manner, $g$ is given so that
its complex dilatation is 
\begin{align*}
\mu_{g}(\zeta)
&=\frac{S_{g}(j(\zeta))(\zeta-j(\zeta))^2 j_{\bar z}(\zeta)}
{2+S_{g}(j(\zeta))(\zeta-j(\zeta))^2 j_z(\zeta)} \quad (\zeta \in \Omega).
\end{align*}
By setting $\zeta =f_{\nu}(z)$, we obtain from \eqref{reflection} that
\begin{align*}
|S_{g}(j(\zeta))(\zeta-j(\zeta))^2 j_{\bar z}(\zeta)|
\leq \frac{1}{\varepsilon}|\varphi(z^*)| \rho_{\mathbf D^*}^{-2}(z^*)<1.
\end{align*}
Hence,
$|\mu_{g}(f_{\nu}(z))| \leq \frac{1}{\varepsilon}\,|\varphi(z^*)| \rho_{\mathbf D^*}^{-2}(z^*)$
for every $z \in \mathbf D$, and thus $\mu_{g} \circ f_{\nu} \in M_p(\mathbf D)$.

We denote the complex dilatation of $g \circ f_{\nu}$ by $\nu_{\varphi}$.  
Formula \eqref{chain} yields
$$
|\nu_{\varphi}(z)-\nu_{\varphi'}(z)| \leq \frac{|\mu_{g} \circ f_{\nu}(z)-\mu_{g'} \circ f_{\nu}(z)|}{\sqrt{(1-\Vert\mu_{g} \Vert_\infty^2)(1-\Vert \mu_{g'}\Vert_\infty^2)}}
$$
for any $\varphi$ and $\varphi'$ in $\mathcal A_p(\mathbf D^*)$ with $\Vert \varphi \Vert_{\mathcal A_p},
\Vert \varphi' \Vert_{\mathcal A_p} <\varepsilon/c_p$,
where $g$ and $g'$ are the corresponding
quasiconformal homeomorphisms. 
As both $\mu_{g} \circ f_{\nu}$ and $\nu=\nu_0$ belong to $M_p(\mathbf D)$, so does $\nu_{\varphi}$.

Because
$\Phi(\nu_\varphi)=S_{g \circ f_{\nu}|_{\mathbf D^*}}=\psi+\varphi$,  
we have a local right inverse $s_\psi$ of $\Phi$ on the neighborhood
$$
V_\psi=\{\psi+\varphi \mid \Vert \varphi \Vert_{\mathcal A_p}<\varepsilon/c_p\} 
\subset \mathcal A_p(\mathbf D^*)
$$
by the correspondence $s_\psi:\psi+\varphi \mapsto \nu_\varphi$. 
By the above 
inequalities for $\mu_{g} \circ f_{\nu}$ and $\nu_{\varphi}$, we see that 
there is a constant $C>0$ such that
$$
|\nu_{\varphi}(z)-\nu_{\varphi'}(z)| \leq C |\varphi(z^*)-\varphi'(z^*)| \rho_{\mathbf D^*}^{-2}(z^*) 
$$
for $z \in \mathbf D$. This implies that $s_\psi$ is continuous.
\end{proof}

\begin{remark}
A similar proof is in \cite[Lem.7.5]{Mat2} for a different kind of Teichm\"uller space.
The necessity of the bi-Lipschitz quasiconformal extension is also pointed out in \cite[p.68]{Sh2}.
\end{remark}

\begin{lemma}\label{holomorphy}
Let $\Psi$ be a holomorphic map on an open subset $V$ of $\mathcal A_\infty(\mathbf D^*)$ into $M(\mathbf D)$.
If $\Psi$ is continuous on 
the open subset $V \cap {\mathcal A}_p(\mathbf D^*)$ 
into $M_p(\mathbf D)$ in the stronger topologies,
then it is holomorphic.
\end{lemma}

\begin{proof}
We apply a claim on infinite-dimensional holomorphy in \cite[p.28]{Bour} (see also \cite[Lem.V.5.1]{Le}).
For $\mu \in M_p(\mathbf D)$, we consider $\alpha_E(\mu)=\int_E \mu(z)dxdy$ for each measurable subset 
$E \subset \mathbf D$ as an element of the dual space of the Banach space of 
all measurable functions $\mu$ on $\mathbf D$ with norm $\Vert \mu \Vert_\infty+\Vert \mu \Vert_p$ finite.
If $\alpha_E(\mu)=0$ for every $E \subset \mathbf D$,
then $\mu=0$. Moreover,
$\alpha_E(\Psi(\varphi)) \in \mathbf C$ for each $E$ depends holomorphically on $\varphi \in V$.
Then, by the claim cited above, we see that $\Psi:V \cap {\mathcal A}_p(\mathbf D^*) \to M_p(\mathbf D)$
is holomorphic. 
\end{proof}

\begin{remark}\label{usual}
Similar applications of this argument can be found in \cite[p.591]{FH}, \cite[p.141]{SW}, and \cite[p.964]{Ya}.
Moreover, we can also deduce the holomorphy of $\Psi$ by showing 
that they are locally bounded and G\^ateaux holomorphic. This is also based on \cite[p.28]{Bour} and formulated in 
\cite[Lem.6.1]{WM-6}.
\end{remark}

For every $[\nu] \in T_p$, we define a map $R_{[\nu]}:T_p \to T_p$ by $R_{[\nu]}([\mu])=\pi(r_\nu(\mu))$ for $[\mu] \in T_p$,
where $r_\nu$ is the homeomorphic automorphism of $M_p(\mathbf D)$ given by
a bi-Lipschitz representative $\nu$. 
This is well defined as the right translation in the group structure of $T_p$.
By taking the local right inverse of the projection $\pi$ at each $[\mu] \in T_p$ as in Theorem \ref{Pcomplex}, 
%and by Lemma \ref{TTholo},
we see that 
$R_{[\nu]}$ is continuous, and hence it is a homeomorphic automorphism of $T_p$.
A similar argument as in Lemma \ref{holomorphy} shows that $R_{[\nu]}$ is in fact a biholomorphic automorphism of $T_p$
identified with the open subset
$\beta(T_p)$ in $\mathcal A_p(\mathbf D^*)$ for $p \geq 1$.

\section{Characterization by Besov spaces}
We intrinsically characterize the quasisymmetric extension $f^\mu|_{\mathbf S}$ of
$f^\mu:\mathbf D \to \mathbf D$ 
with $\mu \in M_p(\mathbf D)$ for $p>1$. 
Hereafter, we change the roles of $\mathbf D$ and $\mathbf D^*$ 
for the sake of the definition of function spaces.

For a holomorphic function $\phi$ on $\mathbf D$, 
we define
$$
\Vert \phi \Vert_{\mathcal B_p}=\left(\int_{\mathbf D} |\phi'(z)|^p(1-|z|^2)^{p-2}dxdy\right)^{\frac{1}{p}}.
$$
The set of all holomorphic functions $\phi$ on $\mathbf D$ with $\Vert \phi \Vert_{\mathcal B_p}<\infty$
is denoted by $\mathcal B_p(\mathbf D)$ and called the {\it analytic $p$-Besov space} on $\mathbf D$.
By ignoring the difference of complex constants, $\mathcal B_p(\mathbf D)$ is a complex Banach space
with this norm. For details, see \cite[Sect.5.3]{Zh}.

Let $H(z)=-i(z+1)/(z-1)$ be the Cayley transformation,  
which maps $\mathbf D$ onto $\mathbf U$.
For a conformal map $f$ of $\mathbf D$,
we take the conjugate $\hat f=H \circ f \circ H^{-1}$.
Moreover, for a holomorphic function $\phi$ on $\mathbf D$,
we take the push-forward $\phi_\ast=\phi \circ H^{-1}$.
Then,
$$
\Vert \phi \Vert_{\mathcal B_p}=\left(\int_{\mathbf U} |\phi_*'(\zeta)|^p(2\,{\rm Im}\,\zeta)^{p-2}d\xi d\eta\right)^{\frac{1}{p}}
=:\Vert \phi_* \Vert_{\mathcal B_p},
$$
and the set of all such holomorphic functions $\phi_*$ on $\mathbf U$ with $\Vert \phi_* \Vert_{\mathcal B_p}<\infty$
is defined to be the analytic $p$-Besov space $\mathcal B_p(\mathbf U)$ on $\mathbf U$.
Similarly, for the Schwarzian derivative $\varphi=S_f$,
let $\varphi_*=S_{\hat f}=S_{f \circ H^{-1}}$ on $\mathbf U$ and let
$\Vert \varphi_* \Vert_{\mathcal A_p}:=\Vert \varphi \Vert_{\mathcal A_p}$.
The Banach space $\mathcal A_p(\mathbf U)=\{\varphi_* \mid \varphi \in \mathcal A_p(\mathbf D)\}$ is defined
by this norm.

\begin{proposition}\label{equivalence}
Let $p>1$. 
The following conditions are equivalent for
a conformal homeomorphism $f$ of $\mathbf D$ extending to
a quasiconformal homeomorphism of $\widehat {\mathbf C}$ with
$f(1)=1$ and $f(\infty)=\infty$:
$\rm{(i)}$
$S_f \in \mathcal A_p(\mathbf D)$;
$\rm{(ii)}$
$\log f' \in \mathcal B_p(\mathbf D)$;
$\rm{(iii)}$
$S_{\hat f} \in \mathcal A_p(\mathbf U)$;
$\rm{(iv)}$
$\log \hat f' \in \mathcal B_p(\mathbf U)$.
\end{proposition}

\begin{proof}
The equivalence $\rm{(i)} \Leftrightarrow \rm{(ii)}$ is proved in \cite[Th.1]{Gu}.
$\rm{(iii)} \Leftrightarrow \rm{(iv)}$ 
%for $p=2$ 
is proved in \cite[Th.4.4]{STW} and \cite[Th.1.3]{LiS}. See \cite[Th.7.1]{WM-4}.
%The argument in the case of $p>1$ is the same.
$\rm{(i)} \Leftrightarrow \rm{(iii)}$ is due to the M\"obius invariance.
\end{proof}

Any function $\phi \in \mathcal B_p(\mathbf D)$ has a non-tangential limit at almost every point of $\mathbf S$
(see \cite[Lem.10.13]{Zh}). This is also true for $\phi_* \in \mathcal B_p(\mathbf U)$. These define
the boundary functions $\phi|_\mathbf S$ on $\mathbf S$ and ${\phi_*}|_{\mathbf R}$ on
$\mathbf R$.

\begin{definition}
A locally integrable complex-valued function $u$ on $\mathbf S$ belongs to the {\it $p$-Besov space}
$B_p(\mathbf S)$ on $\mathbf S$ for $p>1$ if 
$$
\Vert u \Vert_{B_p}=\left(\int_{\mathbf S}\int_{\mathbf S} 
\frac{|u(x_1)-u(x_2)|^p}{|x_1-x_2|^2} |dx_1||dx_2| \right)^{\frac{1}{p}}<\infty.
$$
The $p$-Besov space
$B_p(\mathbf R)$ on $\mathbf R$ is the space of all functions
$u_*=u \circ H^{-1}$ for $u \in B_p(\mathbf S)$ with the norm
$\Vert u_* \Vert_{B_p}:=\Vert u \Vert_{B_p}$.
We regard them as complex Banach spaces by ignoring
additive constants. 
\end{definition}

The next claim says that the boundary function $\phi|_{\mathbf S}$ of 
$\phi \in \mathcal B_p(\mathbf D)$ belongs to $B_p(\mathbf S)$ for $p>1$.
This statement is in \cite[p.131]{Zh}. An explicit proof can be found in \cite[Th.2.1, 5.1]{Pa}
as mentioned in \cite[p.505]{Rei}.

\begin{lemma}\label{boundary}
There is a constant $C_p \geq 1$ depending only on $p>1$ such that for every $\phi \in \mathcal B_p(\mathbf D)$, 
$$C_p^{-1} \Vert \phi \Vert_{\mathcal B_p}  \leq \Vert \phi|_{\mathbf S} \Vert_{B_p} \leq C_p \Vert \phi \Vert_{\mathcal B_p}.$$
\end{lemma}

\begin{remark}
For $p=2$, $\mathcal B_2(\mathbf D)$ is nothing but the analytic Dirichlet space
on $\mathbf D$ and the corresponding $B_2(\mathbf S)$ coincides with the Sobolev space $H^{\frac{1}{2}}(\mathbf S)$.
In this case, the Douglas formula gives the equality 
$\Vert \phi|_{\mathbf S} \Vert_{B_p} = 2\sqrt{\pi} \Vert \phi \Vert_{\mathcal B_p}$
(see \cite[Th.2-5]{Ah2}).
\end{remark}

The above lemma is also valid for $\mathcal B_p(\mathbf U)$ and $\mathcal B_p(\mathbf L)$ defined similarly on
the lower half-plane $\mathbf L$ both of which have the Besov space $B_p(\mathbf R)$ on $\mathbf R$ 
as the space of boundary extensions.

In general, any quasisymmetric homeomorphism $f^\mu|_{\mathbf R}$ for $\mu \in M(\mathbf U)$ can be
represented as conformal welding 
$f^\mu|_{\mathbf R}=g_{\bar \mu^{-1}}^{-1} \circ f_{\mu}|_{\mathbf R}$,
where $f_\mu$ is conformal on $\mathbf L$ and has the complex dilatation $\mu$ on $\mathbf U$, 
and $g_{\bar \mu^{-1}}$ is conformal on $\mathbf U$ and has the complex dilatation $\bar \mu^{-1}$ 
of $(\overline{f^{\mu}})^{-1}$
on $\mathbf L$. 
Here, $\overline{f^{\mu}}:\mathbf L \to \mathbf L$ denotes the reflection of 
$f^{\mu}:\mathbf U \to \mathbf U$.
All the mappings are normalized so that they fix $0$, $1$, and $\infty$.

We intrinsically characterize 
a quasisymmetric homeo\-morphism $f^\mu|_{\mathbf R}$ with $\mu \in M_p(\mathbf U)$.
For $p \geq 2$, this is in \cite[Th.1.1]{Sh}, \cite[Th.1.3]{STW}, and \cite[Th.1.2]{TS}.

\begin{theorem}\label{Besov}
A Beltrami coefficient $\mu \in M(\mathbf U)$ 
represents the Teichm\"uller class $[\mu]$ in $T_p$
%belongs to $M_p(\mathbf U)$ 
for $p>1$
if and only if the quasisymmetric homeomorphism $f^\mu|_{\mathbf R}$ 
is locally absolutely continuous and $\log (f^\mu|_{\mathbf R})'$
is in $B_p(\mathbf R)$. The corresponding claim is true for $M(\mathbf D)$ and $B_p(\mathbf S)$.
\end{theorem}

\begin{proof}
By Lemma \ref{biLip}, we may assume that $f^\mu$ is bi-Lipschitz.
We consider conformal welding 
\begin{equation}\label{welding}
f^\mu|_{\mathbf R}=g_{\bar \mu^{-1}}^{-1} \circ f_{\mu}|_{\mathbf R}.
\end{equation}
Lemma \ref{group} implies that $\bar \mu^{-1} \in M_p(\mathbf L)$.
Then, by Lemma \ref{TTholo} and Proposition \ref{equivalence}, we have $\log (f_\mu)' \in \mathcal B_p(\mathbf L)$
and $\log (g_{\bar \mu^{-1}})' \in \mathcal B_p(\mathbf U)$, and by Lemma \ref{boundary}, 
$\log (f_\mu)'|_{\mathbf R} \in B_p(\mathbf R)$
and $\log (g_{\bar \mu^{-1}})' |_{\mathbf R} \in B_p(\mathbf R)$.

Moreover,
the non-tangential limits of $(f_\mu)'$ and $(g_{\bar \mu^{-1}})'$ coincide with the angular derivatives of
$f_\mu$ and $g_{\bar \mu^{-1}}$, and in fact, the limits for the derivatives are allowed to be taken 
without 
restriction (see \cite[Th.5.5]{Pom}). This implies that
$\log (f_\mu)'|_{\mathbf R}=\log (f_\mu|_{\mathbf R})'$ and
$\log (g_{\bar \mu^{-1}})' |_{\mathbf R}=\log (g_{\bar \mu^{-1}}|_{\mathbf R})'$.

By \eqref{welding},
we see that
$f^\mu|_{\mathbf R}$ is locally absolutely continuous because so are $f_{\mu}|_{\mathbf R}$ and $g_{\bar \mu^{-1}}|_{\mathbf R}$
(see \cite[Lem.3.2]{WM-4})
with $(g_{\bar \mu^{-1}}|_{\mathbf R})'>0$ a.e. Hence,
\begin{equation}\label{log}
\log (g_{\bar \mu^{-1}}|_{\mathbf R} )'\circ f^\mu|_{\mathbf R}+\log (f^\mu|_{\mathbf R})'=\log(f_{\mu}|_{\mathbf R})'.
\end{equation}
The first term of the left side of \eqref{log} also belongs to $B_p(\mathbf R)$
(see \cite[Th.12]{Bo}). Thus, $\log (f^\mu|_{\mathbf R})' \in B_p(\mathbf R)$.

Conversely, \cite[Th.4.5]{WM-3} asserts that 
the variant of Beurling--Ahlfors extension by the heat kernel
due to Fefferman, Kenig and Pipher \cite{FKP} (see also \cite{WM-2}) extends such a quasisymmetric homeomorphism $f$ of $\mathbf R$ with
$\log f' \in B_p(\mathbf R)$ for $p>1$
to a quasiconformal self-homeomorphism of $\mathbf U$
with its complex dilatation in $M_p(\mathbf U)$.

In the unit disk case, the proof for the ``only-if'' part is the same. The ``if'' part
is in \cite[Th.5.2]{WM-3}.
\end{proof}

The correspondence $\mu \mapsto \log (f^\mu|_{\mathbf R})'$ in Theorem \ref{Besov} 
induces an injection
$L:T_p \to B_p^{\mathbf R}(\mathbf R)$, where $B_p^{\mathbf R}(\mathbf R)$ is the real Banach subspace of
$B_p(\mathbf R)$ consisting of real-valued functions.
This is continuous, and in fact, real-analytic
as shown in \cite{WM-4}. 
Conversely, \cite[Th.4.4]{WM-3} proves that the variant of Beurling--Ahlfors extension 
%by the heat kernel 
yields a real-analytic map
$\Lambda:B_p^{\mathbf R}(\mathbf R) \to M_p(\mathbf U)$ such that $\pi \circ \Lambda$ is the inverse of $L$.
%For every $\nu \in \Lambda(B_p^{\mathbf R}(\mathbf R))$, $f^\nu$ is bi-Lipschitz by \cite[Prop.3.2]{WM-3},
%and hence $\pi$ is continuous at $\nu$ by Lemma \ref{group}. 
This shows that
$L$ is a bi-real-analytic homeomorphism. 

\begin{proposition}
The Teichm\"uller space $T_p$ is real-analytically equi\-valent to $B_p^{\mathbf R}(\mathbf R)$ for $p>1$. 
In particular, $T_p$ is 
contractible.
\end{proposition}

\begin{remark}
For $p \geq 2$, the barycentric extension yields a global real-analytic right inverse for the
Teichm\"uller projection $\pi:M_p(\mathbf D) \to T_p$, from which we know that $T_p$ is contractible.
\end{remark}

\end{document}